\documentclass[12pt]{article}
\usepackage[T1]{fontenc}
\usepackage[utf8]{inputenc}
\usepackage{lmodern}
\usepackage{prettyref}
\usepackage{booktabs}
\usepackage[margin=1in]{geometry}
\usepackage{nicefrac}
\usepackage{mathtools}
\usepackage{url}
\usepackage{amsmath}
\usepackage{amsthm}
\usepackage{amsfonts}
\usepackage{amssymb}
\usepackage{graphicx}
\usepackage{xcolor}
\usepackage[numbers]{natbib}
\usepackage[unicode=true,
  bookmarks=true,bookmarksnumbered=false,bookmarksopen=false,
breaklinks=false,pdfborder={0 0 0},backref=false,colorlinks=true]
{hyperref}
\hypersetup{pdftitle={Fast Engset Computation},
pdfauthor={Parsiad Azimzadeh, Tommy Carpenter}}

\makeatletter

\providecommand{\tabularnewline}{\\}

\theoremstyle{plain}
\newtheorem{thm}{\protect\theoremname}
\theoremstyle{plain}
\newtheorem{lem}[thm]{\protect\lemmaname}
\theoremstyle{plain}
\newtheorem{cor}[thm]{\protect\corollaryname}
\theoremstyle{plain}
\newtheorem{conjecture}[thm]{\protect\conjecturename}
\ifx\proof\undefined
\newenvironment{proof}[1][\protect\proofname]{\par
  \normalfont\topsep6\p@\@plus6\p@\relax
  \trivlist
  \itemindent\parindent
\item[\hskip\labelsep\scshape #1]\ignorespaces
}{%
  \endtrivlist\@endpefalse
}
\providecommand{\proofname}{Proof}
\fi

\pdfoutput=1

\definecolor{mylinkcolor}{HTML}{0066cc}
\definecolor{mycitecolor}{HTML}{008800}
\definecolor{myurlcolor}{HTML}{0066cc}

\hypersetup{
  linkcolor  = mylinkcolor,
  citecolor  = mycitecolor,
  urlcolor   = myurlcolor,
  colorlinks = true,
}

\renewcommand*{\leq}{\leqslant}
\renewcommand*{\geq}{\geqslant}

\providecommand{\doi}{}
\renewcommand{\doi}[1]{\href{https://doi.org/#1}{doi:
\begingroup\urlstyle{rm}\nolinkurl{#1}\endgroup}}

\usepackage{microtype}

\theoremstyle{plain}
\newtheorem*{asm*}{\protect\assumptionname}
\theoremstyle{remark}
\newtheorem*{rem*}{\protect\remarkname}

\providecommand{\assumptionname}{Assumption}
\providecommand{\remarkname}{Remark}

\newrefformat{prop}{Proposition \ref{#1}}
\newrefformat{fig}{Figure \ref{#1}}
\newrefformat{tab}{Table \ref{#1}}
\newrefformat{def}{Definition \ref{#1}}
\newrefformat{exa}{Example \ref{#1}}
\newrefformat{cor}{Corollary \ref{#1}}
\newrefformat{conj}{Conjecture \ref{#1}}

\newcommand{\fecabstract}{
  The blocking probability of a finite-source bufferless queue is a
  fixed point of the Engset formula, for which we prove existence and
  uniqueness.
  Numerically, the literature suggests a fixed point iteration.
  We show that such an iteration can fail to converge and is
  dominated by a simple Newton's method, for which we prove a global
  convergence result.
  The analysis yields a new Turán-type inequality involving
hypergeometric functions, which is of independent interest.}

\newcommand{\feckeywords}[1]{\par\noindent
{\small{\em Keywords\/}: #1}}

\newcommand{\fecmsc}[1]{\par\noindent
{\small{\em MSC2010\/}: #1}}

\newcommand{\feccode}[1]{\par\noindent
{\small{\em Code\/}: #1}}

\newcommand{\fecpip}[1]{\par\noindent
{\small{\em Python install via pip\/}: #1}}

\@ifundefined{showcaptionsetup}{}{%
\PassOptionsToPackage{caption=false}{subfig}}
\usepackage{subfig}
\makeatother

\providecommand{\conjecturename}{Conjecture}
\providecommand{\corollaryname}{Corollary}
\providecommand{\lemmaname}{Lemma}
\providecommand{\theoremname}{Theorem}

\begin{document}

\title{Fast Engset computation}

\author{Parsiad Azimzadeh\\
  \small Cheriton School of Computer Science,\\
  \small University of Waterloo\\
  \small Waterloo, Ontario, Canada\\
  \small \href{mailto:pazimzad@uwaterloo.ca} {\tt pazimzad@uwaterloo.ca}
  \and
  Tommy Carpenter\\
  \small Cheriton School of Computer Science,\\
  \small University of Waterloo\\
  \small Waterloo, Ontario, Canada\\
\small \href{mailto:tcarpent@uwaterloo.ca} {\tt tcarpent@uwaterloo.ca}}

\date{}
\maketitle
\begin{abstract}
  \fecabstract
\end{abstract}
\feckeywords{Engset formula, teletraffic, hypergeometric functions,
Turán-type inequality}

\fecmsc{33C05, 90B22} \renewcommand{\qedsymbol}{$\blacksquare$}

\feccode{\url{https://github.com/parsiad/fast-engset/releases}}

\fecpip{\texttt{pip install fast-engset}}

\section{Introduction}

The Engset formula is used to determine the blocking probability in
a bufferless queueing system with a finite population of sources.
Applications to \emph{bufferless optical networks}
\cite{detti2002performance,zukerman2004teletraffic,le2005scalable,verby2005performance,nail2007stochastic}
have sparked a renewed interest in the Engset model and its generalizations
\cite{cohen1957generalized}. Sztrik provides a literature review
of applications \cite{sztrik2001finite}, including \emph{multiprocessor
performance modeling }and the\emph{ machine interference problem},
in which machines request service from one or more repairmen.  The
analysis herein was inspired by a recent application in \emph{sizing
vehicle pools} for car-shares \cite{carpenter2014sizing}.

The queue under consideration is the $M/M/m/m/N$ queue
\cite{kleinrock1975queueing}.\footnote{Subject to some technical
  assumptions, the Engset formula remains
  valid under general distributions (i.e. $G/G/m/m/N$) \cite[Section
5.4]{tijms2003first}.} This is a bufferless queue with $N$
\emph{sources} that can request
service, provided by one of $m$ identical \emph{servers}. When all
$m$ servers are in use, incoming arrivals are \emph{blocked} and
leave the system without queueing. The \emph{Engset formula} is used
to determine the probability $P$ that any random arrival is blocked.
The Engset formula is \cite[Equation (62)]{kubasik1985some}
\begin{equation}
  P=\underbrace{\lim_{P^{\prime}\rightarrow
  P}\frac{\binom{N-1}{m}\left(M(P^{\prime})\right)^{m}}{\sum_{X=0}^{m}\binom{N-1}{X}\left(M(P^{\prime})\right)^{X}}}_{f(P)}\text{
  where }M(P)=\frac{\alpha}{1-\alpha\left(1-P\right)}\tag{Engset
  formula}\label{eq:engset_formula}
\end{equation}
The number of sources $N$, the number of servers $m$, and the\emph{
}offered traffic\emph{ per-source\emph{}\footnote{Some sources
    represent the \ref{eq:engset_formula} using the \emph{total
    }offered traffic $E=N\alpha$ in lieu of $\alpha$. In this case,
$M(P)=E/(N-E(1-P))$.}} $\alpha$ are given as input.

It is not obvious if any value of $P$ satisfies the \ref{eq:engset_formula},
or if multiple values of $P$ might satisfy it. To the authors' best
knowledge, this work is the first to establish the existence and uniqueness
of a solution (Section \ref{sec:properties_of_the_engset_formula}).

\begin{rem*}

  The limit appearing in the \ref{eq:engset_formula} is a technical
  detail to avoid (for ease of analysis) the removable discontinuity
  at $P=1-1/\alpha$. We mention that $f$ may admit nonremovable discontinuities
  at some negative values of $P$ (at which the limit does not exist),
  though this does not affect the analysis.

\end{rem*}

\begin{rem*}

  Let $\lambda$ be the \emph{idle source initiation rate}, the rate
  at which a free source (i.e. one not being serviced) initiates requests,
  and $1/\mu$ be the \emph{mean service time}. If $P$ is the blocking
  probability, $M(P)=\lambda/\mu$. This substitution removes $P$ from
  the right-hand side of the \ref{eq:engset_formula} \cite[Equation
  (70)]{kubasik1985some}.
  However, $\lambda$ is often unknown in practice, and hence this method
  is only applicable in special cases, or subject to error produced
  from approximating $\lambda$.

\end{rem*}

\section{\label{sec:properties_of_the_engset_formula}Properties of the Engset
formula}

If the number of servers $m$ is zero, any request entering the queue
is blocked ($P=1$). If there are at least as many servers as there
are sources $(m\geq N)$, any request entering the queue can immediately
be serviced ($P=0$). Finally, the case of zero traffic $(\alpha=0)$
corresponds to a queue that receives no requests. We assume the following
for the remainder of this work:

\begin{asm*}

  $m$ and $N$ are integers with $0<m<N$. $\alpha$ is a positive
  real number.

\end{asm*}

The following lemmas characterize $f$ defined in the \ref{eq:engset_formula}
and are used to establish several results throughout this work:
\begin{lem}
  \label{lem:f_is_strictly_decreasing}$f$ is strictly decreasing on
  $[0,\infty)$.
\end{lem}

\begin{lem}
  \label{lem:f_is_convex}$f$ is convex on
  $[1-1/\alpha,\infty)\supset[1,\infty)$.
\end{lem}
Owing partly to \prettyref{lem:f_is_strictly_decreasing}, our first
significant result is as follows:
\begin{thm}
  \label{thm:existence_and_uniqueness}There exists a unique probability
  $P^{\star}$ satisfying the \ref{eq:engset_formula}.
\end{thm}
Proofs of these results are given in Appendix \ref{sec:proofs}.
The proof of \prettyref{thm:existence_and_uniqueness} establishes
that $f(0)-0$ and $f(1)-1$ have opposite signs. Therefore, $P^{\star}$
can be computed via the bisection method on the interval $[0,1]$
applied to the map
\begin{equation}
  P\mapsto f(P)-P.\label{eq:root_map}
\end{equation}

\section{Computation}

\subsection{Fixed point iteration}

The literature suggests the use of a fixed point iteration \cite[page
489]{keshav1997engineering}.
This involves picking an initial guess $P_{0}$ for the blocking probability
and considering the iterates of $f$ evaluated at $P_{0}$. Specifically,
\begin{align}
  P_{0} & \in\left[0,1\right]\nonumber \\
  P_{n} & =f(P_{n-1})\text{ for }n>0.\tag{fixed point
  iteration}\label{eq:fixed_point_iteration}
\end{align}
We characterize convergence in the following result:
\begin{thm}
  \label{thm:fixed_point_convergence_tight}If $\alpha\leq1$ and
  $|f^{\prime}(0)|<1$,
  the \ref{eq:fixed_point_iteration} converges to $P^{\star}$.
\end{thm}
While the first inequality appearing above is a restriction on the
per-source traffic, the second inequality is hard to verify, as it
involves the derivative of $f$. This inspires the following:
\begin{cor}
  \label{cor:fixed_point_convergence}If $\alpha\leq1$ and $N\geq2m$,
  the \ref{eq:fixed_point_iteration} converges to $P^{\star}$.
\end{cor}
The condition $N\geq2m$ requires there to be twice as many sources
as there are servers, satisfied in most (but not all) reasonable queueing
systems.

Proofs of these results are given in Appendix \ref{sec:proofs}.

\subsection{Newton's method}

\emph{Newton's method} uses first-derivative information in an attempt
to speed up convergence. In particular,
\begin{align}
  P_{0} & \in\left[0,1\right]\nonumber \\
  P_{n} &
  =P_{n-1}-\frac{f(P_{n-1})-P_{n-1}}{f^{\prime}(P_{n-1})-1}\text{ for
  }n>0.\tag{Newton's method}\label{eq:newtons_method}
\end{align}
Often, convergence results for applications of Newton's method are
\emph{local} in nature: they depend upon the choice of initial guess
$P_{0}$. By using the convexity established in \prettyref{lem:f_is_convex},
we are able to derive a \emph{global} result for Newton's method:
\begin{thm}
  \label{thm:newtons_method}If $\alpha\leq1$, \ref{eq:newtons_method}
  converges to $P^{\star}$.
\end{thm}
A proof of this result is given in Appendix \ref{sec:proofs}. Superficially,
\prettyref{thm:newtons_method} seems preferable to
\prettyref{cor:fixed_point_convergence}
as it does not place restrictions on $N$ or $m$. In practice, we
will see that \ref{eq:newtons_method} outperforms the
\ref{eq:fixed_point_iteration},
and that it performs well even when $\alpha>1$ (Section
\ref{sec:numerical_results}).

\section{\label{sec:numerical_results}Comparison of methods}

\prettyref{tab:table_arxiv} compares methods  for a queueing
system with $N=20$ sources (though we mention that the observed trends
seem to hold independent of our choice of $N$). The initial guess
used is $P_{0}=\nicefrac{1}{2}$. The stopping criterion used is
$|P_{n+1}-P_{n}|\leq tol=2^{-24}$.

\begin{table}
  \begin{centering}
    \subfloat[$\alpha=\nicefrac{1}{4}$\label{tab:results_point_25-1}]{
      \begin{centering}
        {\scriptsize%
          \begin{tabular}{cccc}
            \toprule
            Servers & Probability & \multicolumn{2}{c}{Number of
            iterations}\tabularnewline
            \cmidrule{3-4}
            $m$ & $P^{\star}$ & Fixed point & Newton\tabularnewline
            \midrule
            1 & 8.322e-01 & 6 & 3\tabularnewline
            2 & 6.725e-01 & 7 & 3\tabularnewline
            3 & 5.235e-01 & 7 & 3\tabularnewline
            4 & 3.879e-01 & 8 & 3\tabularnewline
            5 & 2.693e-01 & 9 & 3\tabularnewline
            6 & 1.714e-01 & 8 & 4\tabularnewline
            7 & 9.718e-02 & 8 & 4\tabularnewline
            8 & 4.753e-02 & 7 & 4\tabularnewline
            9 & 1.947e-02 & 6 & 4\tabularnewline
            10 & 6.554e-03 & 5 & 3\tabularnewline
            11 & 1.798e-03 & 4 & 3\tabularnewline
            12 & 4.005e-04 & 4 & 3\tabularnewline
            13 & 7.194e-05 & 3 & 3\tabularnewline
            14 & 1.028e-05 & 3 & 3\tabularnewline
            15 & 1.142e-06 & 3 & 3\tabularnewline
            16 & 9.518e-08 & 3 & 2\tabularnewline
            17 & 5.599e-09 & 2 & 2\tabularnewline
            18 & 2.074e-10 & 2 & 2\tabularnewline
            19 & 3.638e-12 & 2 & 2\tabularnewline
            \bottomrule
        \end{tabular}}
        \par
      \end{centering}

    }\subfloat[$\alpha=\nicefrac{1}{2}$\label{tab:results_point_2-1}]{
      \begin{centering}
        {\scriptsize%
          \begin{tabular}{cccc}
            \toprule
            Servers & Probability & \multicolumn{2}{c}{Number of
            iterations}\tabularnewline
            \cmidrule{3-4}
            $m$ & $P^{\star}$ & Fixed point & Newton\tabularnewline
            \midrule
            1 & 9.087e-01 & 7 &  3\tabularnewline
            2 & 8.187e-01 & 8 &  3\tabularnewline
            3 & 7.303e-01 & 9 &  3\tabularnewline
            4 & 6.436e-01 & 10 &  3\tabularnewline
            5 & 5.591e-01 & 11 &  3\tabularnewline
            6 & 4.773e-01 & 11 &  3\tabularnewline
            7 & 3.985e-01 & 14 &  3\tabularnewline
            8 & 3.235e-01 & 15 &  4\tabularnewline
            9 & 2.531e-01 & 16 &  4\tabularnewline
            10 & 1.885e-01 & 16 &  4\tabularnewline
            11 & 1.310e-01 & 14 &  4\tabularnewline
            12 & 8.259e-02 & 12 &  4\tabularnewline
            13 & 4.527e-02 & 10 &  4\tabularnewline
            14 & 2.041e-02 & 8 &  4\tabularnewline
            15 & 7.124e-03 & 6 &  4\tabularnewline
            16 & 1.827e-03 & 5 &  4\tabularnewline
            17 & 3.254e-04 & 4 &  3\tabularnewline
            18 & 3.623e-05 & 3 &  3\tabularnewline
            19 & 1.907e-06 & 3 &  3\tabularnewline
            \bottomrule
        \end{tabular}}
        \par
      \end{centering}

    }
    \par
  \end{centering}

  \begin{centering}
    \subfloat[$\alpha=1$\label{tab:results_1-1}]{
      \begin{centering}
        {\scriptsize%
          \begin{tabular}{cccc}
            \toprule
            Servers & Probability & \multicolumn{2}{c}{Number of
            iterations}\tabularnewline
            \cmidrule{3-4}
            $m$ & $P^{\star}$ & Fixed point & Newton\tabularnewline
            \midrule
            1 & 9.523e-01 & 7 & 3\tabularnewline
            2 & 9.047e-01 & 8 & 3\tabularnewline
            3 & 8.574e-01 & 10 & 3\tabularnewline
            4 & 8.102e-01 & 12 & 3\tabularnewline
            5 & 7.633e-01 & 14 & 4\tabularnewline
            6 & 7.166e-01 & 17 & 4\tabularnewline
            7 & 6.702e-01 & 20 & 4\tabularnewline
            8 & 6.241e-01 & 25 & 4\tabularnewline
            9 & 5.782e-01 & 33 & 4\tabularnewline
            10 & 5.327e-01 & 45 & 3\tabularnewline
            11 & 4.874e-01 & 79 & 3\tabularnewline
            12 & 4.424e-01 & 556 & 4\tabularnewline
            13 & 3.976e-01 & \texttt{FAIL} & 4\tabularnewline
            14 & 3.530e-01 & \texttt{FAIL} & 4\tabularnewline
            15 & 3.084e-01 & \texttt{FAIL} & 5\tabularnewline
            16 & 2.636e-01 & \texttt{FAIL} & 5\tabularnewline
            17 & 2.181e-01 & \texttt{FAIL} & 6\tabularnewline
            18 & 1.708e-01 & \texttt{FAIL} & 7\tabularnewline
            19 & 1.187e-01 & \texttt{FAIL} & 7\tabularnewline
            \bottomrule
        \end{tabular}}
        \par
      \end{centering}

    }\subfloat[$\alpha=2$\label{tab:results_2-1}]{
      \begin{centering}
        {\scriptsize%
          \begin{tabular}{cccc}
            \toprule
            Servers & Probability & \multicolumn{2}{c}{Number of
            iterations}\tabularnewline
            \cmidrule{3-4}
            $m$ & $P^{\star}$ & Fixed point & Newton\tabularnewline
            \midrule
            1 & 9.756e-01 & 7 & 3\tabularnewline
            2 & 9.512e-01 & 9 & 3\tabularnewline
            3 & 9.268e-01 & 10 & 3\tabularnewline
            4 & 9.025e-01 & 13 & 4\tabularnewline
            5 & 8.781e-01 & 15 & 4\tabularnewline
            6 & 8.538e-01 & 19 & 4\tabularnewline
            7 & 8.295e-01 & 24 & 4\tabularnewline
            8 & 8.053e-01 & 33 & 4\tabularnewline
            9 & 7.810e-01 & 54 & 4\tabularnewline
            10 & 7.568e-01 & 136 & 4\tabularnewline
            11 & 7.325e-01 & \texttt{FAIL} & 4\tabularnewline
            12 & 7.083e-01 & \texttt{FAIL} & 4\tabularnewline
            13 & 6.840e-01 & \texttt{FAIL} & 4\tabularnewline
            14 & 6.597e-01 & \texttt{FAIL} & 4\tabularnewline
            15 & 6.353e-01 & \texttt{FAIL} & 4\tabularnewline
            16 & 6.107e-01 & \texttt{FAIL} & 4\tabularnewline
            17 & 5.859e-01 & \texttt{FAIL} & 5\tabularnewline
            18 & 5.604e-01 & \texttt{FAIL} & 5\tabularnewline
            19 & 5.336e-01 & \texttt{FAIL} & 5\tabularnewline
            \bottomrule
        \end{tabular}}
        \par
      \end{centering}

    }
    \par
  \end{centering}

  \caption{Comparison under $N=20$. \texttt{FAIL} indicates
  divergence.\label{tab:table_arxiv}}
\end{table}

Bisection halves the search interval at each step, so that the maximum
possible error at the $n$-th iteration is $2^{-n}$. It follows that
to achieve a desired error tolerance $tol$, bisection requires
$\lceil-\lg(tol)\rceil=\lceil-\lg(2^{-24})\rceil=24$
iterations independent of the input parameters (for this reason, it
is omitted from the tables). The \ref{eq:fixed_point_iteration} fails
to converge or performs poorly (sometimes taking hundreds of iterations)
precisely when the sufficient conditions of
\prettyref{cor:fixed_point_convergence}
are violated. \ref{eq:newtons_method} outperforms both algorithms
by a wide margin, often converging in just a few iterations.

Insight into the poor performance of the \ref{eq:fixed_point_iteration}
is given by \prettyref{cor:fixed_point_upper_bound} of Appendix
\ref{sec:proofs}, which exploits the oscillatory nature of the
\ref{eq:fixed_point_iteration}
(see Figure \ref{fig:oscillatory}) to derive successively tighter
upper bounds on the number of iterations required for convergence
up to a desired error tolerance.

\begin{figure}
  \centering
  \subfloat[$m=10$, $N=2m$, $\alpha=1$, and $P_{0}=\nicefrac{1}{2}$:
  convergent.]{%
    \begin{minipage}[t]{0.47\textwidth}
      \centering
      \includegraphics[width=0.95\linewidth]{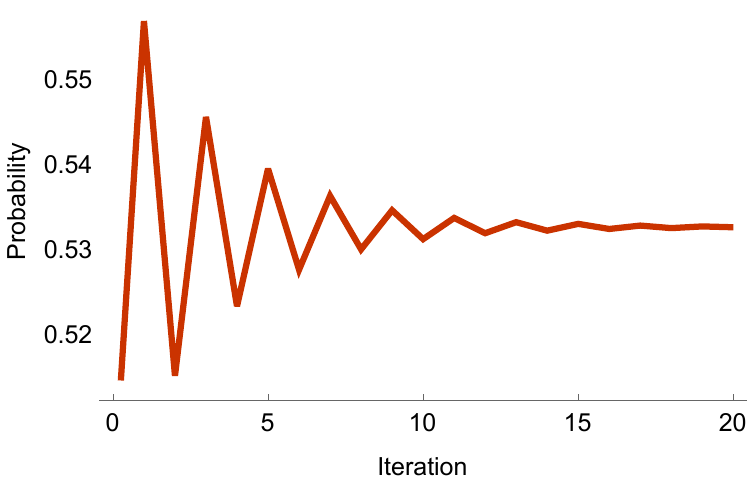}
  \end{minipage}}%
  \hfill
  \subfloat[$m=15$, $N=20<2m$, $\alpha=1$, and
  $P_{0}=\nicefrac{1}{2}$: divergent.]{%
    \begin{minipage}[t]{0.47\textwidth}
      \centering
      \includegraphics[width=0.95\linewidth]{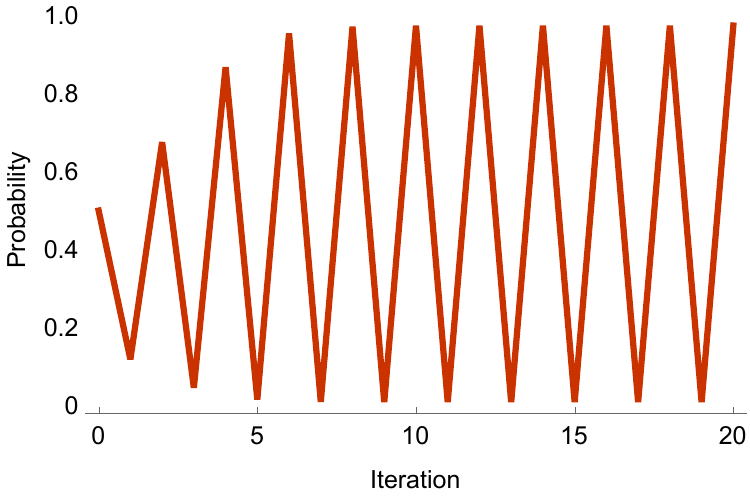}
  \end{minipage}}
  \caption{\label{fig:oscillatory}Oscillatory nature of the
  \ref{eq:fixed_point_iteration}.}
\end{figure}

\begin{rem*}Naïve implementations computing $f$ (and $f^{\prime}$)
  directly may take more iterations than necessary due to floating point
  error. \prettyref{lem:f_is_a_hypergeometric_reciprocal} of Appendix
  \ref{sec:proofs} shows that $f$ is a reciprocal of a hypergeometric
  function so that standard computational techniques
  \cite{pearson2009computation}
  can be used. A\emph{ quasi-Newton} implementation has been made available
  by the authors: \url{https://github.com/parsiad/fast-engset/releases}.

\end{rem*}

\section{\label{sec:turan}A Turán-type inequality}

Turán-type inequalities are named after Paul Turán, who proved the
result $(L_{n}(x))^{2}>L_{n-1}(x)L_{n+1}(x)$ on $-1<x<1$ for the
Legendre Polynomials $\{L_{n}\}$. Such inequalities appear frequently
for hypergeometric functions and are often a direct consequence of
their log-concavity/convexity. There exists a maturing body of work
characterizing the log-concavity/convexity and associated Turán-type
inequalities of generalized hypergeometric functions (see, e.g.,
\cite{baricz2008functional,baricz2008turan,karp2010log,kalmykov2013log}).

The analysis used to prove \prettyref{lem:f_is_convex} gives rise
to a new Turán-type inequality. Letting $_{2}F_{1}$ denote the ordinary
hypergeometric function \cite{abramowitz1972handbook}, we have the
following result, whose proof is given in Appendix \ref{sec:proofs}:
\begin{thm}[A Turán-type inequality]
  \label{thm:turan_type_inequality}Let $b$ be a positive integer,
  $c$ a positive real number, and
  \[
    h_{n}(x)={}_{2}F_{1}(1+n,-b+n;c+n;-x).
  \]
  Then, the map $x\mapsto h_{1}(x)/(h_{0}(x))^{2}$ is strictly decreasing
  on $[0,\infty)$ and
  \begin{equation}
    b\left(c+1\right)\cdot\left(h_{1}(x)\right)^{2}\geq\left(b-1\right)c\cdot
    h_{0}(x)h_{2}(x)\text{ for }x\geq0.\label{eq:turan_type_inequality}
  \end{equation}
  When $b=1$, the right-hand side of \eqref{eq:turan_type_inequality}
  is understood to be zero.

\end{thm}

\section{Future work}

Numerical evidence suggests that \prettyref{lem:f_is_convex} can
be relaxed:
\begin{conjecture}
  \label{conj:f_is_convex}$f$ is convex on $[0,\infty)$.
\end{conjecture}
This result would remove the requirement $\alpha\leq1$ from all claims
in this work. In particular, this would yield unconditional convergence
for \ref{eq:newtons_method}.

\appendix

\section{\label{sec:proofs}Proofs of results}

Let $(\cdot)_{X}$ denote the\emph{ Pochhammer symbol}:
\begin{equation}
  \left(c\right)_{X}=
  \begin{cases}
    c\left(c+1\right)\cdots\left(c+X-1\right), & \text{if }X\text{ is
    a positive integer};\\
    1, & \text{if }X=0.
  \end{cases}\tag{Pochhammer symbol}\label{eq:pochhammer_symbol}
\end{equation}
The ordinary hypergeometric function \cite{abramowitz1972handbook}
satisfies
\begin{equation}
  _{2}F_{1}(a,b;c;z)=\sum_{X=0}^{\infty}\frac{\left(a\right)_{X}\left(b\right)_{X}}{\left(c\right)_{X}}\frac{z^{X}}{X!}\text{
  if }b\in\left\{ 0,-1,-2,\ldots\right\} \text{ or
  }\left|z\right|<1\tag{hypergeometric
  function}.\label{eq:hypergeometric_function}
\end{equation}
The \ref{eq:pochhammer_symbol} can also be used to represent the
\emph{falling factorial} $c^{(X)}$:
\begin{multline*}
  c^{\left(X\right)}=c\left(c-1\right)\cdots\left(c-X+1\right)\\
  =\left(-c\right)\left(-1\right)\left(-c+1\right)\left(-1\right)\cdots\left(-c+X-1\right)\left(-1\right)=\left(-c\right)_{X}\left(-1\right)^{X}.
\end{multline*}

\begin{lem}
  \label{lem:f_is_a_hypergeometric_reciprocal}$f(P)$ defined in the
  \ref{eq:engset_formula} satisfies
  \[
    1/f(P)={}_{2}F_{1}(1,-m;N-m;1-P-1/\alpha).
  \]
\end{lem}
\begin{proof}
  If $P=1-1/\alpha$, the claim is trivial. Otherwise, the reciprocal
  of $M(P)$ in the \ref{eq:engset_formula} is
  \begin{equation}
    1/M(P)=-\left(1-P-1/\alpha\right).\label{eq:reciprocal_of_M}
  \end{equation}
  We can write the binomial coefficients in the \ref{eq:engset_formula}
  in terms of \ref{eq:pochhammer_symbol}s as follows:
  \begin{equation}
    \binom{N-1}{X}/\binom{N-1}{m}=\frac{m!}{X!}\frac{\left(N-1-m\right)!}{\left(N-1-\phantom{m}\mathllap{X}\right)!}=\frac{m^{\left(m-X\right)}}{\left(N-m\right)_{m-X}}.\label{eq:binomial_to_pochhammer}
  \end{equation}
  Substituting \eqref{eq:reciprocal_of_M} and \eqref{eq:binomial_to_pochhammer}
  into the reciprocal of $f(P)$ yields
  \begin{multline*}
    \frac{1}{f(P)}=\sum_{X=0}^{m}\frac{m^{\left(m-X\right)}}{\left(N-m\right)_{m-X}}\left(1/M(P)\right)^{m-X}=\sum_{X=0}^{m}\frac{m^{\left(X\right)}}{\left(N-m\right)_{X}}\left(1/M(P)\right)^{X}\\
    =\sum_{X=0}^{m}\frac{\left(-m\right)_{X}}{\left(N-m\right)_{X}}\left(1-P-1/\alpha\right)^{X}=\sum_{X=0}^{\infty}\frac{\left(-m\right)_{X}}{\left(N-m\right)_{X}}\left(1-P-1/\alpha\right)^{X}.
  \end{multline*}
  The upper bound of summation is relaxed to $\infty$ in the last
  equality since $(-m)_{X}=0$ if $X>m$. The desired result then follows
  from multiplying each summand in the series by $(1)_{X}/X!=1$.
\end{proof}
The following identity should be understood subject to the convention
$0=0\cdot\infty=\infty\cdot0$ ($\infty$ denotes complex infinity):
\begin{lem}[Hypergeometric binomial theorem]
  \label{lem:hypergeometric_binomial}Suppose $b$ is a nonpositive integer
  and $c$ is not an integer satisfying $b\leq c\leq0$. Then,
  \[
    _{2}F_{1}(a,b;c;z+w)=\sum_{Y=0}^{\infty}\frac{\left(a\right)_{Y}\left(b\right)_{Y}}{\left(c\right)_{Y}}\frac{z^{Y}}{Y!}{}_{2}F_{1}(a+Y,b+Y;c+Y;w).
  \]
\end{lem}
\begin{proof}
  An application of the binomial theorem yields
  \begin{multline*}
    _{2}F_{1}(a,b,c;z+w)=\sum_{X=0}^{\infty}\frac{\left(a\right)_{X}\left(b\right)_{X}}{\left(c\right)_{X}}\frac{\left(z+w\right)^{X}}{X!}\\
    =\sum_{X=0}^{\infty}\frac{\left(a\right)_{X}\left(b\right)_{X}}{\left(c\right)_{X}}\frac{1}{X!}\sum_{Y=0}^{X}\binom{X}{Y}z^{Y}w^{X-Y}=\sum_{Y=0}^{\infty}\frac{z^{Y}}{Y!}\sum_{X=Y}^{\infty}\frac{\left(a\right)_{X}\left(b\right)_{X}}{\left(c\right)_{X}}\frac{w^{X-Y}}{\left(X-Y\right)!}\\
    =\sum_{Y=0}^{\infty}\frac{\left(a\right)_{Y}\left(b\right)_{Y}}{\left(c\right)_{Y}}\frac{z^{Y}}{Y!}\sum_{X=Y}^{\infty}\frac{\left(a+Y\right)_{X-Y}\left(b+Y\right)_{X-Y}}{\left(c+Y\right)_{X-Y}}\frac{w^{X-Y}}{\left(X-Y\right)!}.
  \end{multline*}
  The desired result follows by shifting the index of summation to
  $X=0$.
\end{proof}
\prettyref{lem:hypergeometric_binomial} can also be extended to the
case where $b$ is not a negative integer, but care must be taken
to ensure that the various power series are convergent.

\begin{proof}[Proof of \prettyref{lem:f_is_strictly_decreasing}]
  To establish this, we show that $P\mapsto1/f(P)$ is a polynomial
  with positive coefficients. That is,
  \begin{equation}
    \frac{1}{f(P)}=\sum_{Y=0}^{m}c_{Y}P^{Y}\text{ where
    }c_{Y}>0.\label{eq:reciprocal_of_f_as_polynomial}
  \end{equation}
  An application of \prettyref{lem:hypergeometric_binomial} to the
  form in \prettyref{lem:f_is_a_hypergeometric_reciprocal} reveals
  that
  \begin{equation}
    c_{Y}=\frac{m^{\left(Y\right)}}{\left(N-m\right)_{Y}}d_{Y}.\label{eq:reciprocal_of_f_explicit_coefficient_form}
  \end{equation}
  where $d_{Y}={}_{2}F_{1}(1+Y,-(m-Y);N-m+Y;1-1/\alpha)$. To arrive
  at \eqref{eq:reciprocal_of_f_as_polynomial}, it suffices to show
  $d_{Y}>0$. Another application of \prettyref{lem:hypergeometric_binomial}
  along with the identity
  \[
    _{2}F_{1}(a,-b;c;1)=\frac{\left(c-a\right)_{b}}{\left(c\right)_{b}}\text{
    if }b\text{ is a nonnegative integer}
  \]
  yields
  \begin{equation}
    d_{Y}=\sum_{Z=0}^{m-Y}\frac{\left(1/\alpha\right)^{Z}}{Z!}\frac{\left(1+Y\right)_{Z}\left(m-Y\right)^{\left(Z\right)}}{\left(N-m+Y\right)_{Z}}\frac{\left(N-m-1\right)_{m-Y-Z}}{\left(N-m+Y+Z\right)_{m-Y-Z}},\label{eq:reciprocal_of_f_polynomial_subsubcoefficient}
  \end{equation}
  which is trivially positive.
\end{proof}

\begin{rem*}A concise proof of $d_{Y}>0$ for the case of $\alpha>1$
  ($\alpha\leq1$ is trivial) is given by the Euler transform:
  $_{2}F_{1}(a,b;c;z)=(1-z)^{c-a-b}{}_{2}F_{1}(c-a,c-b;c;z)$.
\end{rem*}

The following is adapted from \cite[Lemma 1]{karp2010log}:
\begin{lem}
  \label{lem:polynomial_quotient_decreasing}Let
  \[
    A(Q)=\sum_{X=0}^{N}a_{X}Q^{X}\text{ and }B(Q)=\sum_{X=0}^{N}b_{X}Q^{X}
  \]
  be polynomials with nonnegative coefficients satisfying
  $a_{X}b_{X-1}\leq a_{X-1}b_{X}$
  for $0<X\leq N$ and $b_{X}>0$ for $0\leq X\leq N$. Then, the quotient map
  $Q\mapsto A(Q)/B(Q)$ is nonincreasing on $[0,\infty)$.
  If, in addition, there is no constant $\lambda$ such that
  $A=\lambda B$, then the quotient map is strictly decreasing on $[0,\infty)$.
\end{lem}

\begin{proof}[Proof of \prettyref{lem:f_is_convex}]
  The derivative of the \ref{eq:hypergeometric_function} is
  \begin{equation}
    \frac{\partial}{\partial
    z}\,{}_{2}F_{1}(a,b;c;z)=\frac{ab}{c}{}_{2}F_{1}(a+1,b+1;c+1;z).\label{eq:hypergeometric_derivative}
  \end{equation}
  This fact combined with the representation in
  \prettyref{lem:f_is_a_hypergeometric_reciprocal}
  yields
  \begin{equation}
    f^{\prime}(P)=-\frac{m}{N-m}\frac{A(P+1/\alpha-1)}{B(P+1/\alpha-1)}\label{eq:f_prime}
  \end{equation}
  where
  \begin{align*}
    A(Q) & ={}_{2}F_{1}(2,-\left(m-1\right);N-m+1;-Q)\\
    \text{and }B(Q) & =\left(_{2}F_{1}(1,-m;N-m;-Q)\right)^{2}.
  \end{align*}
  To arrive at the desired result, we seek to show that the map
  \begin{equation}
    Q\mapsto A(Q)/B(Q)\label{eq:fixed_point_convergence_decreasing_map}
  \end{equation}
  is strictly decreasing on $[0,\infty)$.

  For notational succinctness, let $S=N-m$. We can write
  \eqref{eq:fixed_point_convergence_decreasing_map}
  as a quotient of polynomials by noting that
  \[
    A(Q)=\sum_{X=0}^{m-1}\frac{\left(X+1\right)\left(m-1\right)^{\left(X\right)}}{\left(S+1\right)_{X}}Q^{X}
  \]
  and (expanding using the Cauchy product)
  \[
    B(Q)=\left(\sum_{X=0}^{m}\frac{m^{\left(X\right)}}{\left(S\right)_{X}}Q^{X}\right)^{2}=\sum_{X=0}^{2m}Q^{X}\sum_{Y=0}^{X}\frac{m^{\left(Y\right)}}{\left(S\right)_{Y}}\frac{m^{\left(X-Y\right)}}{\left(S\right)_{X-Y}}.
  \]
  We seek to apply \prettyref{lem:polynomial_quotient_decreasing} on
  the polynomials $A$ and $B$, whose coefficients we denote $a_{X}$
  and $b_{X}$, respectively. Note that their coefficient vectors are
  not scalar multiples since $a_{0}=b_{0}=1$, while
  $0=a_{X}<b_{X}$ for $m\leq X\leq2m$. One can easily check that
  $a_{1}=a_{1}b_{0}\leq a_{0}b_{1}=b_{1}$.
  We thus need only verify $a_{X}b_{X-1}\leq a_{X-1}b_{X}$ for $X>1$.

  Fix $X>1$. It is easy to check that
  \[
    a_{X}=a_{X-1}\left(\frac{1}{X}+1\right)\frac{m-X}{S+X}.
  \]
  Using Gauss summation, we can rewrite $b_{X}$ as
  \[
    b_{X}=\mathbf{1}_{\{X\text{ is
    even}\}}\left(\frac{m^{\left(X/2\right)}}{\left(S\right)_{X/2}}\right)^{2}+2\sum_{Y=0}^{\left\lfloor
      \left(X-1\right)/2\right\rfloor
    }\frac{m^{\left(Y\right)}}{\left(S\right)_{Y}}\frac{m^{\left(X-Y\right)}}{\left(S\right)_{X-Y}}.
  \]
  Suppose $X$ is even. Then,
  \begin{align*}
    a_{X}b_{X-1} &
    =2a_{X-1}\left(\frac{1}{X}+1\right)\frac{m-X}{S+X}\left(\sum_{Y=0}^{X/2-1}\frac{m^{\left(Y\right)}}{\left(S\right)_{Y}}\frac{m^{\left(X-Y-1\right)}}{\left(S\right)_{X-Y-1}}\right)\\
    &
    \leq2a_{X-1}\left(\frac{1}{X}\frac{m-X}{S+X}\sum_{Y=0}^{X/2-1}\left(\frac{m^{\left(X/2-1\right)}}{\left(S\right)_{X/2-1}}\frac{m^{\left(X/2\right)}}{\left(S\right)_{X/2}}\right)+\sum_{Y=0}^{X/2-1}\frac{m^{\left(Y\right)}}{\left(S\right)_{Y}}\frac{m^{\left(X-Y\right)}}{\left(S\right)_{X-Y}}\right)\\
    & \leq
    a_{X-1}\left(\left(\frac{m^{\left(X/2\right)}}{\left(S\right)_{X/2}}\right)^{2}+2\sum_{Y=0}^{X/2-1}\frac{m^{\left(Y\right)}}{\left(S\right)_{Y}}\frac{m^{\left(X-Y\right)}}{\left(S\right)_{X-Y}}\right)\\
    & \leq a_{X-1}b_{X}.
  \end{align*}
  A similar approach can be taken if $X$ is odd.
\end{proof}

\begin{proof}[Proof of \prettyref{thm:existence_and_uniqueness}]
  We first show that $f(1)<1$, or equivalently, $1/f(1)>1$. By
  the positivity of \eqref{eq:reciprocal_of_f_polynomial_subsubcoefficient},
  the map $\alpha\mapsto1/f(P;\alpha)$ is strictly decreasing. Passing
  to the limit and dropping higher order terms involving $1/\alpha^{Z}$
  with $Z>0$ yields
  \[
    \frac{1}{f(P;\alpha)}>\lim_{\alpha^{\prime}\rightarrow\infty}\frac{1}{f(P;\alpha^{\prime})}=\sum_{Y=0}^{m}P^{Y}\frac{\left(m\right)^{\left(Y\right)}}{\left(N-m\right)_{Y}}\frac{\left(N-m-1\right)_{m-Y}}{\left(N-m+Y\right)_{m-Y}}.
  \]
  One can verify that if $P=1$, the above sum is exactly one, yielding
  $1/f(1)>1$ (for all $0<\alpha<\infty$), as desired.

  By \prettyref{lem:f_is_strictly_decreasing}, the map \eqref{eq:root_map}
  is strictly decreasing on $[0,\infty)$. Furthermore, since $f(1)<1$,
  $f(1)-1$ and $f(0)-0>0$ have opposite signs. Because \eqref{eq:root_map}
  is also continuous, the desired result follows by the intermediate
  value theorem.
\end{proof}

\begin{proof}[Proof of \prettyref{thm:fixed_point_convergence_tight}]
  Let $I=[0,1]$. \eqref{eq:reciprocal_of_f_as_polynomial} establishes
  that $f$ is positive on $I$. Since $\alpha\leq1$, $d_{0}$ appearing
  in \eqref{eq:reciprocal_of_f_explicit_coefficient_form} satisfies
  $d_{0}\geq1$. It follows that $f(0)=1/d_{0}\leq1$ (see
  \eqref{eq:reciprocal_of_f_as_polynomial}).
  Letting $I=[0,1]$, these facts and \prettyref{lem:f_is_strictly_decreasing}
  yield $f(I)\subset I$. Since $f$ is continuously differentiable
  on $I$, it suffices to show that there exists a nonnegative constant
  $L<1$ such that $|f^{\prime}|\leq L$ on $I$ (implying that $f$
  is a contraction on $I$).

  Since $P+1/\alpha-1\geq0$ whenever $\alpha\leq1$, \eqref{eq:f_prime}
  reveals that $-f^{\prime}=|f^{\prime}|$ on $I$. Owing to
  \prettyref{lem:f_is_convex},
  $f$ is convex on $I$ so that $-f^{\prime}$ is nonincreasing on
  $I$. Therefore, $|f^{\prime}(0)|\geq|f^{\prime}|$ on $I$, and the
  desired result follows by taking $L=|f^{\prime}(0)|$.
\end{proof}

\begin{proof}[Proof of \prettyref{cor:fixed_point_convergence}]
  We begin by considering the case of $\alpha<1$; $\alpha=1$ is handled
  separately. Recall that the proof of \prettyref{lem:f_is_convex}
  shows that map \eqref{eq:fixed_point_convergence_decreasing_map}
  is strictly decreasing on $[0,\infty)$. Since $1/\alpha-1>0$ and
  $_{2}F_{1}(a,b;c;0)=1$,
  \begin{equation}
    \left|f^{\prime}(0)\right|=\frac{m}{N-m}\frac{A(1/\alpha-1)}{B(1/\alpha-1)}<\frac{m}{N-m}\frac{A(0)}{B(0)}=\frac{m}{N-m},\label{eq:fixed_point_convergence_strict_inequality}
  \end{equation}
  and the desired result follows ($N\geq2m$ is equivalent to $m/(N-m)\leq1$).

  Suppose now $\alpha=1$. We modify our approach, as the strict inequality
  in \eqref{eq:fixed_point_convergence_strict_inequality} no longer
  holds. By \eqref{eq:reciprocal_of_f_as_polynomial} and
  \eqref{eq:reciprocal_of_f_explicit_coefficient_form},
  \[
    f(0)=1/{}_{2}F_{1}(1,-m,N-m;0)=1.
  \]
  This along with the fact that $f$ is strictly decreasing
  (\prettyref{lem:f_is_strictly_decreasing})
  and $0<f(1)<1$ implies that the iterates of $f$ evaluated at some
  probability $P_{0}$ (i.e. $f^{k}(P_{0})$ for $k>0$) reside in $[f(1),1]$.
  We can thus relax the sufficient condition for convergence in
  \prettyref{thm:fixed_point_convergence_tight}
  to $|f^{\prime}(f(1))|<1$ in lieu of $|f^{\prime}(0)|<1$. Then
  \[
    \frac{m}{N-m}\frac{A(f(1))}{B(f(1))}<\frac{m}{N-m}\frac{A(0)}{B(0)}=\frac{m}{N-m},
  \]
  and the desired result follows.
\end{proof}
Let $f^{k}$ denote the $k$-th iterate of $f$. The proof above reveals
that we can replace the condition $|f^{\prime}(0)|<1$ with
$|f^{\prime}(f^{2k}(0))|<1$
for some nonnegative integer $k$. Owing to this, we derive a relaxation
of \prettyref{thm:fixed_point_convergence_tight} along with a family
of bounds (parameterized by $k$) on the number of iterations required
for convergence up to a desired error tolerance $\epsilon$:
\begin{cor}
  \label{cor:fixed_point_upper_bound}Let $k$ be a nonnegative integer
  and $P^{\star}$ denote the solution of the \ref{eq:engset_formula}.
  Suppose $\alpha\leq1$ and
  \[
    q=\left|f^{\prime}(f^{2k}(0))\right|<1.
  \]
  Given $0<\epsilon\leq1$ and $\{P_{n}\}$ as defined by the
  \ref{eq:fixed_point_iteration},
  $|P_{2k+\ell}-P^{\star}|\leq\epsilon$ whenever
  \[
    \ell\geq\left\lceil \log_{q}(\epsilon-\epsilon q)\right\rceil .
  \]
\end{cor}
\begin{proof}
  Let $u_{n}=f^{n}(0)$. As in the proof of
  \prettyref{thm:fixed_point_convergence_tight}, $f([0,1])\subset[0,1]$.
  Since $f$ is strictly decreasing (\prettyref{lem:f_is_strictly_decreasing}),
  the iterates of $0$ bracket the fixed point from alternating sides:
  \[
    u_{0}\leq u_{2}\leq u_{4}\leq\cdots\leq P^{\star}\leq\cdots\leq
    u_{5}\leq u_{3}\leq u_{1}.
  \]
  The idea is to use these iterates as guardrails. After $2k$ iterations,
  the sequence $\{P_{n}\}$ is trapped in an interval whose left endpoint
  is $u_{2k}$.
  Since $|f^{\prime}|$ is nonincreasing, this endpoint gives
  the desired Lipschitz constant.

  Define
  \[
    I_{k}=
    \begin{cases}
      [0,1], & k=0,\\
      [u_{2k},u_{2k-1}], & k>0.
    \end{cases}
  \]
  We first show that $P_{2k}\in I_{k}$. This is immediate if $k=0$.
  If $k>0$, then $f^{2k}$ is increasing, $f^{2k-1}$ is decreasing,
  and $f(P_{0})\geq0$, so
  \[
    u_{2k}=f^{2k}(0)\leq f^{2k}(P_{0})=P_{2k}=f^{2k-1}(f(P_{0}))\leq
    f^{2k-1}(0)=u_{2k-1}.
  \]
  Thus, $P_{2k}\in I_{k}$.

  Next, we show that $I_{k}$ is invariant under $f$. If $k=0$, this is just
  $f([0,1])\subset[0,1]$. If $k>0$, then the chain above gives
  $u_{2k+1}\leq u_{2k-1}$ and hence
  \[
    f(I_{k})=[u_{2k},u_{2k+1}]\subset[u_{2k},u_{2k-1}]=I_{k}.
  \]
  Therefore, $P_{2k+\ell}\in I_{k}$ for all $\ell\geq0$.

  Since $|f^{\prime}|$ is nonincreasing on $[0,1]$, for every $P\in I_{k}$,
  \[
    |f^{\prime}(P)|\leq |f^{\prime}(u_{2k})|=q.
  \]
  Hence $f$ is a contraction on $I_{k}$ with Lipschitz constant $q$. By
  the contraction mapping principle \cite{palais2007simple},
  \[
    |P_{2k+\ell}-P^{\star}|\leq q^{\ell}|P_{2k}-P^{\star}|\leq q^{\ell}.
  \]
  Since $\ell\geq\left\lceil \log_{q}(\epsilon-\epsilon q)\right\rceil$
  and $0<q<1$, we have $q^{\ell}\leq\epsilon(1-q)\leq\epsilon$, and the
  desired result follows.
\end{proof}
The proof of \prettyref{thm:newtons_method} requires the following
result (a simple modification of \cite[chapter 22, exercise 14b]{spivak1994}):
\begin{lem}
  \label{lem:newtons_method_for_convex_functions}Let $I$ be an interval
  and $g\colon I\rightarrow\mathbb{R}$ be a convex and differentiable
  function satisfying $g^{\prime}<0$ and $g(x^{\star})=0$ for some
  $x^{\star}$ in $I$. Then, given $x_{0}\in I$ with $g(x_{0})\geq0$,
  the sequence $\{x_{n}\}$ defined by
  \[
    x_{n}=x_{n-1}-g(x_{n-1})/g^{\prime}(x_{n-1})\text{ for }n>0
  \]
  converges from below (i.e. $x_{0}\leq x_{1}\leq\cdots$) to $x^{\star}$.
\end{lem}
\begin{proof}
  Since $g(x_{0})\geq0$ and $g^{\prime}(x_{0})<0$, it follows that
  $x_{0}\leq x_{1}$. Since $(x_{1},0)$ is on a tangent line of $g$
  and a convex function lies above its tangent lines, $g(x_{1})\geq0$.
  Hence, $x_{1}\leq x^{\star}$. Repeating this argument establishes
  $x_{0}\leq x_{1}\leq\cdots\leq x^{\star}$.

  It follows that $x_{n}\rightarrow x$ for some $x$ in $I$. Taking
  limits on both sides of $g^{\prime}(x_{n-1})(x_{n-1}-x_{n})=g(x_{n-1})$
  and using the facts that $g$ is continuous and $g^{\prime}$ is monotone
  due to the assumption of convexity, we arrive at $g(x)=0$. Since
  a strictly decreasing function cannot have two distinct roots, $x=x^{\star}$.
\end{proof}

\begin{proof}[Proof of \prettyref{thm:newtons_method}]
  First, consider the case $f(P_{0})-P_{0}\geq0$.
  \prettyref{lem:f_is_strictly_decreasing}
  implies that $f^{\prime}\leq0$ and hence $f^{\prime}-1<0$ on $I=[0,1]$.
  \prettyref{lem:f_is_convex} establishes that $f$ is convex on $I$.
  \prettyref{thm:existence_and_uniqueness} guarantees the existence
  of $P^{\star}$ in $I$ such that $f(P^{\star})-P^{\star}=0$. Letting
  $g\colon I\rightarrow\mathbb{R}$ be defined by $g(P)=f(P)-P$, we
  can directly apply \prettyref{lem:newtons_method_for_convex_functions}.

  Now, consider the case of $f(P_{0})-P_{0}<0$. Note that
  \[
    P_{1}=P_{0}-\frac{f(P_{0})-P_{0}}{f^{\prime}(P_{0})-1}=\frac{f(P_{0})+P_{0}\left|f^{\prime}(P_{0})\right|}{1+\left|f^{\prime}(P_{0})\right|}\in\left[0,1\right].
  \]
  Since the point $(P_{1},0)$ is on a tangent line of $P\mapsto f(P)-P$
  and a convex function lies above its tangent lines, $f(P_{1})-P_{1}\geq0$.
  We can now repeat the argument in the first paragraph with the initial
  guess $P_{1}$ in lieu of $P_{0}$.
\end{proof}

\begin{proof}[Proof of \prettyref{thm:turan_type_inequality}]
  If $b=1$, then $h_{0}(x)=1+x/c$ and $h_{1}(x)=1$.
  Therefore, $x\mapsto h_{1}(x)/(h_{0}(x))^{2}=(1+x/c)^{-2}$ is
  strictly decreasing on $[0,\infty)$.
  Moreover, the right-hand side of \eqref{eq:turan_type_inequality}
  is zero, so the inequality is immediate.
  Assume henceforth that $b>1$.

  That $x\mapsto h_{1}(x)/(h_{0}(x))^{2}$ is strictly decreasing follows
  directly from the proof of \prettyref{lem:f_is_convex}. The derivative
  of this map is $C(x)/(h_{0}(x))^{3}$ where
  \[
    C(x)=-2h_{1}(x)h_{0}^{\prime}(x)+h_{0}(x)h_{1}^{\prime}(x)=-\frac{2b}{c}\left(h_{1}(x)\right)^{2}+\frac{2\left(b-1\right)}{c+1}h_{0}(x)h_{2}(x)
  \]
  (the last equality is a consequence of \eqref{eq:hypergeometric_derivative}).
  Since $h_{0}$ is positive on $H=[0,\infty)$, it follows that $C$
  is nonpositive on $H$, yielding \eqref{eq:turan_type_inequality}.
\end{proof}
\textbf{Acknowledgements}: We thank D. B. Karp and M. Scheuer for
enlightening discussions on the topic of special functions. We thank
S. Keshav and J. W. Wong for advice and direction.

\bibliographystyle{plainnat}
\bibliography{main}

\end{document}